\documentclass[12pt,twoside]{amsart}
\usepackage{amsmath, amsthm, amscd, amsfonts, amssymb, float, graphicx, color
, soul }
\usepackage[bookmarksnumbered, 
, plainpages]{hyperref}
\usepackage{graphicx}
\usepackage{epstopdf}
\usepackage{rotating}
\usepackage{mathtools}
\input xy
\xyoption{all}
\newtheorem{theorem}{Theorem}
\usepackage[bottom]{footmisc}
\newtheorem{lemma}[theorem]{Lemma}
\newtheorem{proposition}[theorem]{Proposition}

\newtheorem{example}[theorem]{Example}

\def\pn{\par\noindent}

\def\CA{{\rm CA}}
\def\cent{{\rm Cent}}
\def\diam{{\rm diam}}
\date{}
\title{On $n$-centralizer $\CA$-groups}
\author{Mohammad A. Iranmanesh$^*$, Mohammad Hossein Zareian}
\thanks{{\scriptsize
\hskip -0.4 true cm MSC(2010): Primary: 20E99, Secondary: 05C25
$*$Corresponding author}}
\begin{document}
\maketitle

\begin{abstract}
Let $G$ be a finite non-abelian group and $m=|G|/|Z(G)|$. In this paper we investigate $m$-centralizer group $G$ with cyclic center and
we will prove that if $G$ is a finite non-abelian $m$-centralizer $\CA$-group,
then there exists an integer $r>1$ such that $m=2^r.$ It is also prove that if $G$ is an $m$-centralizer non-abelian finite group
which is not a $\CA$-group and its derived subgroup $G'$ is of order 2, then there exists an integer $s>1$ such that $m=2^{2s}.$
\end{abstract}

\section{Introduction}
Throughout this paper all groups are assumed to be finite. By $Z(G)$, $C_G(x)$, $\cent(G)$ and $x^G$ we denote the
center of the group $G$, the centralizer of $x$, the set of centralizers of the group $G$ and the conjugacy class of $x\in G$ respectively.
A group $G$ is called {\it $n$-centralizer} if $|\cent(G)|=n.$
The influence of $|\cent(G)|$ on the group $G$ has investigated in \cite{AJH, A, AT, BS, JR}.
It is clear by definition that a group $G$ is 1-centralizer if and only if it is abelian. Belcastro and Sherman show that
there is no finite $n$-centralizer groups for $n\in \{2,3\}.$

The {\it non-commuting graph} $\Gamma(G)$ with respect to $G$ is a graph with vertex set $G\setminus Z(G)$
and two distinct vertices $x$ and $y,$ are adjacent whenever $[x, y]\neq 1.$
It is proved in \cite{AAM} that $\Gamma(G)$ is connected with diameter $2$ and girth $3$.
A non-abelian group $G$ is called a {\it $\CA$-group} if $C_G(x)$ is abelian for all $x\in G\setminus Z(G)$. We
refer the readers to see \cite{S} for more properties of this class of finite groups.
A group $G$ is called an {\it extra-especial $p$-group} if $G/Z(G)$ is an elementary abelian $p$-group
and $|Z(G)|=p$ where $p$ is a prime number.

Brady in his Ph.D. thesis \cite{BR} consider the finite nilpotent groups of class two with cyclic centre and
his results were further exploited by Brady et al. and Leong (see \cite{BBC,L}).
\begin{theorem} \cite[Theorem 2.2]{BBC} The $q$-groups of class 2 on two generators with cyclic
centre comprise the following list:
\[ Q(n,r)=\langle a,b: a^{q^{n}}=b^{q^{r}}=1, a^{q^{n-r}}=[a,b]\rangle;\] 
$2r\leq n;$
\begin{align*} Q(n,r)=\langle a,b: a^{q^{n}}=b^{q^{r}}=1, a^{q^{r}}=[a,b]^{q^{2r-n}}, [a,b,a]=[a,b,b]=1\rangle;\end{align*}
$r\leq n <2r,$

and if $q=2$ we have as well
\begin{align*} R(n)=\langle a,b: a^{2^{n+1}}=b^{2^{n+1}}=1, a^{2^{n}}=[a,b]^{2^{n-1}}=b^{2^{n}}, [a,b,a]=[a,b,b]=1\rangle,\end{align*}
$n\geq 1$.
\end{theorem}
Note that $Q(n, 0)$ is in fact the cyclic group of order $q^n, n > 0$.

\begin{lemma}\label{L3}\cite[Theorems A, B]{L}
\begin{enumerate}
\item[1)] Every finite 2-group $G$ of class two with cyclic center, either has the central decomposition:
\[G\cong Q(n_1,r_1)...Q(n_{\alpha}, r_{\alpha}) Q(\ell,\ell)^{\epsilon_l}...Q(1,1)^{\epsilon_1},\]
where $\alpha \geq 0, \epsilon_i \geq 0, i=1,...,\ell,$
\[n_1>...> n_{\alpha} > \ell\geq 1,~ n_{\alpha} > r_1 >...> r_{\alpha} \geq 0, ~1 < n_1 - r_1< ...< n_{\alpha}-r_{\alpha},\]
or else it has the central decomposition:
$$G\cong R(n) Q(\ell,\ell)^{\epsilon_{\ell}} ... Q(1,1)^{\epsilon_1},$$
where $n\geq \ell\geq 1, \epsilon_{i} \geq 0, i=1,...,\ell.$
\item[2)] The above canonical decomposition is unique up to isomorphism.
\end{enumerate}
\end{lemma}
The main purpose of this paper is to study $m$-centralizer $\CA$-groups where $m=|G|/|Z(G)|$ and $m\neq 2, 3$.
We will prove the following theorems.

\begin{theorem}\label{T1}
Let $G$ be a non-abelian group. The following are equivalent
\begin{enumerate}
\item[1)] $G$ is an $m$-centralizer $\CA$-group and $m\neq 2, 3$
\item[2)] $G=A\times P$ where $P$ is a 2-group, $P$ is a $\CA$-group and $|\cent(P)|=|P|/|Z(P)|.$
\item[3)] $G=A\times P$ where $C_P(x)=Z(P)\cup xZ(P),$ for all $x\in P\setminus Z(P).$
\end{enumerate}
\end{theorem}
\begin{theorem}\label{T2}
Let $G$ be a non-abelian $m$-centralizer group and $Z(G)$ is cyclic. Then
\begin{enumerate}
\item[1)] $G=A\times P$,  $A\cong \mathbb{Z}_t, t$ is odd and $P$ is a 2-group with
cyclic center and the following central decomposition:
$$Q(n,0)Q(1,1)^{\epsilon}~where~\epsilon\geq 1, n> 1~or$$
$$Q(n,1)Q(1,1)^{\epsilon}~where~\epsilon\geq 0, n> 2~or$$
$$Q(1,1)^{\epsilon}~where~\epsilon> 1,~or~R(1)Q(1,1)^{\epsilon}~where~\epsilon\geq 0.$$
\item[2)] If $G$ is a $\CA$-group, then $G=A\times P$, $A\cong \mathbb{Z}_t, t$ is odd and $P$ is a 2-group
with cyclic center which has the following central decomposition:
$$Q(n,0)Q(1,1), n> 1~or~Q(n,1),n>2~or~Q(1,1)~or~R(1).$$
\end{enumerate}

\end{theorem}
\begin{theorem} \label{T3}
Let $G$ be an $m$-centralizer $\CA$-group. Then there exist integer $r>1$ such that
$m=2^r.$ Conversely for an arbitrary integer $r>1$, there exist a $\CA$-group $G$ which is $2^r$-centralizer.
\end{theorem}
\begin{theorem} \label{T4}
Let $G$ be a non-abelian $m$-centralizer group which is not a $\CA$-group. Then
\begin{enumerate}
\item[1)] there exists normal subgroup $H\unlhd G$ so that $|H|=16|Z(G)|, Z(G)=Z(H)$.
\item[2)] Also if $|G'|=2$, then $m=2^{2s}$, for some integer $s>1$.
\end{enumerate}
\end{theorem}

Section 2 contains some definitions and preliminaries. In Section 3 we prove the main theorems.
\section{Definitions and preliminaries}
We consider two equivalence relations on the group $G$ that will be used in the proofs. We say $x{\sim_1} y$ if and only
if $C_G(x)=C_G(y).$ Also $x{\sim_2} y$ if and only if $xZ(G)=yZ(G).$ The number of equivalence classes of
${\sim_1}$ and ${\sim_2}$ on the group $G$ are equal with $|\cent(G)|$ and $|G|/|Z(G)|$ respectively.
We denote the equivalence class including $x$ by $[x]_{\sim}.$

We show that $|\cent(G)|=|G|/|Z(G)|$ if and only if $[x]_{\sim_1}=[x]_{\sim_2}$ for all $x\in G.$
Now we present some useful Lemmas.

\begin{lemma} \label{L1}
\cite[Lemma 3.6]{AAM} Let $G$ be a group. Then the following are equivalent:
\begin{enumerate}
\item[1)] $G$ is a $\CA$-group.
\item[2)] If $[x,y]=1$ then $C_G(x)=C_G(y)$ where $x,y\in G\setminus Z(G).$
\item[3)] If $[x,y]=[x,z]=1$ then $[y,z]=1$ where $x\in G\setminus Z(G).$
\item[4)] If $A,B\leq G, Z\lneqq C_G(A) \leqslant C_G(B) \lneqq G,$ then $C_G(A)=C_G(B).$
\end{enumerate}
\end{lemma}

\begin{lemma}\label{L2} \cite[Proposition 2.6]{AAM} Let $G$ be a finite non-abelian group and $\Gamma(G)$
be a regular graph. Then $G$ is nilpotent of class at most $3$ and $G=A\times P,$ where $A$ is an abelian
group and $P$ is a $p$-group ($p$ is a prime) and furthermore $\Gamma(P)$ is a regular graph.
\end{lemma}

\begin{lemma}\label{L4}
\cite[Lemma 4]{MO} Let $p$ be a prime number and let $r>1$ be an integer. Then there exists a
non-abelian $p$-group $P$ of order $p^{2r}$ such that:
\begin{enumerate}
\item[1)] $|Z(P)|=p^r;$
\item[2)] $P/Z(P)$ is an elementary abelian $p$-group;
\item[3)] for every noncentral element $x$ of $P$, $C_P(x)=Z(P)\langle x\rangle.$
\end{enumerate}
\end{lemma}

\begin{lemma}\label{L5}\cite[Theorem 2.1]{JMR}
Let $G$ be a non-abelian group and $|\cent(G)|=|G|/|Z(G)|$. Then $G/Z(G)$ is an elementary abelian 2-group.
\end{lemma}

\section{Proof of the main theorems}
We start with some lemmas which will be used in the proof of the main theorems.

\begin{lemma} \label{L6}
Let $G$ be a non-abelian group. Then $[x]_{\sim_2}\subseteq [x]_{\sim_1},$ for all $x\in G.$
\end{lemma}
\begin{proof} Let $y\in~[x]_{\sim_2}.$ Therefore there exists $z\in Z(G)$ so that $y=xz.$ Since $C_G(y)=C_G(xz)=C_G(x)$ we have $y\in~[x]_{\sim_1}.$
\end{proof}

By Lemma \ref{L6}, $|\cent(G)|\leq |G|/|Z(G)|$. Hence $|\cent(G)|=|G|/|Z(G)|$ if and only if $[x]_{\sim_1}=[x]_{\sim_2},$ for all $x\in G.$

\begin{lemma}\label{L7}
Let $G$ be a non-abelian group. Then the following are equivalent:
\begin{itemize}
\item[1)] If $[x,y]=1$, then $xZ(G)=yZ(G)$ where $x,y\in~G\setminus Z(G).$
\item[2)] $G$ is a $\CA$-group and $[x]_{\sim_1}=[x]_{{\sim_2}},$ for all $x\in G\setminus Z(G).$
\item[3)] $[x,y]=1$ and $[x,w]=1$ imply that $yZ(G)=wZ(G)$ where $x,y,w\in G\setminus Z(G).$
\end{itemize}
\end{lemma}
\begin{proof}
$1)\Rightarrow 2)$ Let $[x, y]=1$ implies that $xZ(G)=yZ(G)$ where $x,y\in G\setminus Z(G).$ Let $x\in G\setminus Z(G)$
and $y\in [x]_{\sim_1}.$ Then $C_G(x)=C_G(y)$ and $[x, y]=1$. By our assumption we have $xZ(G)=yZ(G)$ which means that
$y\in [x]_{\sim_2}.$ Therefore $[x]_{\sim_1}\subseteq [x]_{\sim_2}$ and by Lemma ~\ref{L6} $[x]_{\sim_1}=[x]_{\sim_2}$.
It is enough to prove that $G$ is a $\CA$-group. To this end, let $x,y\in G\setminus Z(G)$ and $[x, y]=1$. Since $xZ(G)=yZ(G)$ we have $[x]_{\sim_2}=[y]_{\sim_2}$.
Since $[x]_{\sim_1}=[x]_{\sim_2}$ it follows $[x]_{\sim_1}=[y]_{\sim_1}$ and therefore $C_G(x)=C_G(y).$ Now by Lemma ~\ref{L1} $G$ is a $\CA$-group.

$2)\Rightarrow 3)$ Let $G$ be a $\CA$-group and $[x]_{\sim_1}=[x]_{\sim_2},$ for all $x\in G\setminus Z(G).$ Let $[x,y]=1$ and $[x,w]=1$
where $x,y,w \in G\setminus Z(G)$. Since $G$ is a $\CA$-group, we have, by Lemma \ref{L1}, that $[y,w]=1$ and $C_G(y)=C_G(w)$.
So $[y]_{\sim_1}=[w]_{\sim_1}.$ By assumption $[y]_{\sim_2}=[w]_{\sim_2}$, which implies that $yZ(G)=wZ(G).$

$3)\Rightarrow 1)$ Let $[x,y]=1$ and $[x,w]=1$ imply that $yZ(G)=wZ(G)$ where $x,y,w\in G\setminus Z(G).$
Assume $[x,y]=1$ and let $w=xz$ for $z\in Z(G)$. Then $w\notin Z(G)$ and $[x,w]=1$. Also $[x,y]=1$.
Hence by assumption $yZ(G)=wZ(G)$. Now $yZ(G)=xzZ(G)$ and we have $yZ(G)=xZ(G)$.
\end{proof}

\begin{lemma}\label{L8}
Let $G$ be a non-abelian group. Let $[x]_{\sim_1}$ and $[y]_{\sim_1}$ be two different classes of ${\sim_1}.$
If $[x_0, y_0]\neq 1$ where $x_0\in [x]_{\sim_1}$ and $y_0\in [y]_{\sim_1},$ then $[u,v]\neq 1$ for all $u\in [x]_{\sim_1}$ and
$v\in [y]_{\sim_1}$. Also $[x_1, x_2]=1$ for all $x_1, x_2\in [x]_{\sim_1}$.
\end{lemma}
\begin{proof}
If $u,v\in [x]_{\sim_1}$ then $C_G(u)=C_G(v)$ so $[u, v]=1$. Let $u\in [x]_{\sim_1},~v\in [y]_{\sim_1},~[u, v]\neq 1$.
If $w\in [y]_{\sim_1}$ then $C_G(w)=C_G(y)=C_G(v)$. Since $u\notin C_G(v)$, $u\notin C_G(w)$ and $[u, w]\neq 1.$
Similarly, it is easy to see that for all $z\in [x]_{\sim_1}$ and for all $w\in [y]_{\sim_1}$, we have $[z, w]\neq 1$.
\end{proof}

\begin{lemma}\label{L9}
Let $G$ be a non-abelian group and $|G'|=2.$ Then
\begin{itemize}
\item[1)] $[x]_{\sim_1}=[x]_{\sim_2}, x\in G;$
\item[2)] there exist $x,y\in G$ such that $[x, y]\neq 1$ and $G=C_G(x)\cup C_G(y)\cup C_G(xy);$
\item[3)] $G=A\times P$ where $P$ is a 2-group and $P/Z(P)$ is an elementary abelian 2-group.
\end{itemize}
\end{lemma}
\begin{proof}
$1)$ Let $G'=\{1,t\}, x\in G$ and $x\neq 1,t.$ Then $[t,x]\in G'$ and either
$[t,x]=1$ or $[t,x]=t$. If $[t,x]=t$, then $t^{-1}x^{-1}tx=t$ and we obtain $x^{-1}tx=t^2=1$.
Hence $tx=x$ and $t=1$ which is a contradiction. Hence $[t,x]=1$ and $t\in Z(G)$. Now we show
that $[x]_{\sim_1}~\subseteq [x]_{\sim_2}$.
For this, let $y\in [x]_{\sim_1}.$ Then $C_G(x)=C_G(y).$ Suppose, towards a contradiction,
that $xZ(G)\neq yZ(G)$. Then $xy^{-1}\notin Z(G)$. Therefore there exists $g\in G\setminus Z(G)$
such that $[xy^{-1},g]=t$. If $[x,g]=1$ and $[y,g]=1,$ then $[xy^{-1},g]=1$
which is a contradiction. So $[x,g]\neq 1$ or $[y, g]\neq 1$. Let $[x,g]\neq 1$. Then $[x,g]=t$ and $gx=xgt$.
Since $[xy^{-1},g]=t, gxy^{-1}=xy^{-1}gt$, it follows that $xgty^{-1}=xy^{-1}gt$ and $gty^{-1}=y^{-1}gt$. Hence
$gy^{-1}=y^{-1}g$ and $g\in C_G(y^{-1})=C_G(y)$, which is another contradiction. Similarly $[y,g]\neq 1$ gives a
contradiction. Therefore $xZ(G)=yZ(G)$ for $y\in [x]_{\sim_1}$ which implies $[x]_{\sim_1}\subseteq [x]_{\sim_2}$.
Now by Lemma \ref{L6}, we have $[x]_{\sim_1}=[x]_{\sim_2}$.

$2)$ Since $G$ is non-abelian, there exist $x, y\in G$ such that $[x,y]\neq 1$. Let $t\in G\setminus Z(G)$ and
$t\notin (C_G(x)\cup C_G(y))$. Then there are $a\neq 1\neq b$, such that $[t,y]=a$ and $[t,x]=b$. This
implies $a, b\in G'$, $a=b=z$ and $|z|=2$. By part (1), $[x]_{\sim_1}=[x]_{\sim_2}$ and
by Lemma \ref{L5}, $G/Z(G)$ is an elementary abelian 2-group. Hence $G'$ is a subgroup of $Z(G)$ and $z\in Z(G)$.
An easy calculation shows that $txy=xyt$. Hence $t\in C_G(xy).$ Therefore $G=C_G(x)\cup C_G(y)\cup C_G(xy)$.

$3)$ Since $|G'|=2$, we have $|x^G|=2$ for all $x\in G\setminus Z(G)$. As $|C_G(x)|=|x^G||G|$,
we find that $\Gamma(G)$ is a regular graph. By Lemma \ref{L2},
$G=A\times P$ where $P$ is an elementary abelian $p$-group. Now, Lemma \ref{L5}, implies $G/Z(G)$ is an
elementary abelian 2-group. Hence $P$ is a 2-group and $P/Z(P)$ is an elementary abelian 2-group. Note that
the fact that $G$ is the direct product of a $2$-group and of an abelian group of odd order can
be found in \ref{MI}.
\end{proof}

\begin{lemma} \label{L10}
Let $G$ be a non-abelian group. Then $C_G(x)=Z(G)\cup xZ(G),$ for all $x\in G\setminus Z(G)$ if and only if $G$ is a $\CA$-group
and $[x]_{\sim_1}=[x]_{\sim_2},$ for all $x\in G.$
\end{lemma}

\begin{proof}
Let $C_G(x)=Z(G)\cup xZ(G)$ for all $x\in G\setminus Z(G).$ Then $C_G(x)$ is abelian for all $x\in G\setminus Z(G).$ So $G$ is a $\CA$-group.
Let $y\in [x]_{\sim_1}.$ Therefore $y\in C_G(x)\setminus Z(G)$ which implies $y\in xZ(G)$. Hence $[x]_{\sim_1} \subseteq [x]_{\sim_2}.$
By Lemma \ref{L6} $[x]_{\sim_1}= [x]_{\sim_2}.$

Conversely let $G$ be a $\CA$-group and $[x]_{\sim_1}=[x]_{\sim_2},$ for all $x\in G.$ clearly $Z(G)\cup xZ(G)\subseteq C_G(x)$.
Suppose $y\in C_G(x)\setminus Z(G).$ Then $[x,y]=1$ and by Lemma \ref{L7}(2), $xZ(G)=yZ(G)$ which implies that $y\in xZ(G)$.
Hence $C_G(x)\subseteq Z(G)\cup xZ(G)$ and therefore $C_G(x)=Z(G)\cup xZ(G)$.
\end{proof}

\begin{example}
The dihedral group $D_8$ is a $\CA$-group and $C_{D_8}(x)=Z(D_8)\cup xZ(D_8)$.
\end{example}

\begin{lemma}\label{L11}
Let $G$ be a non-abelian group. Then the following are equivalent:
\begin{itemize}
\item[1)] $|G|=4|Z(G)|.$
\item[2)] $G$ is a $\CA$-group, $[x]_{\sim_1}=[x]_{\sim_2}$ for all $ x\in G\setminus Z(G)$ and $|G'|=2$.
\item[3)] $G=A\times P$, where $A$ is an abelian group, $P$ is both a 2-group and a $\CA$-group and $|P'|=2$.
\end{itemize}
\end{lemma}

\begin{proof}
$1)\Rightarrow 2)$ Let $|G|=4|Z(G)|$. Since $G$ is non-abelian there exist $x,y\in G$ such that $[x, y]\neq 1$.
So $xxy\neq xyx,yxy\neq xyy.$ Hence $G=Z(G)\cup xZ(G)\cup yZ(G)\cup xyZ(G).$ So $C_G(t)=Z(G)\cup tZ(G)$ for
$t\in G\setminus Z(G).$ By Lemma \ref{L10}, $G$ is a $\CA$-group and $[x]_{\sim_1}=[x]_{\sim_2}$ for all $x\in G$.
Since $G$ is non-abelian and $|G|=4|Z(G)|$ we find that $G/Z(G)$ is an elementary abelian 2-group. Hence $G'$ is an
elementary abelian 2-group. Therefore if $[x,y]=t$ then $o(t)=2$. Let $a,b\in G\setminus Z(G)$,
then $a$ and $b$ are one of $xz, yz, xyz$ for some $z\in Z(G)$. So $[a,b]=[x,y]=t$. Therefore $G'=\{1,t\}$ and $|G'|=2$.

$2)\Rightarrow 3)$ Let $G$ be a $\CA$-group, $[x]_{\sim_1}=[x]_{\sim_2}$ for all $ x\in G\setminus Z(G)$ and
$|G'|=2$. Then $|x^G|=2, x\in G\setminus Z(G).$ By Lemma \ref{L2}, $G=A\times P$ where $A$ is an abelian group and
$P$ is a 2-group. Since $G$ is a $\CA$-group $P$ is also a $\CA$-group, because $C_{A\times P}(a,x)=A\times C_P(x)$
where $(a,x)\in A\times P.$ Since $|G'|=2$ it follows $|P'|=2.$

$3)\Rightarrow 1)$ Let $G=A\times P$, where $A$ is an abelian group, $P$ is a 2-group, $P$ is a
$\CA$-group and $|P'|=2.$ Since $|P'|=2$, by Lemma \ref{L9}, we have $[x]_{\sim_1}=[x]_{\sim_2}$ for all
$x\in P$.
By Lemma \ref{L10}, $C_P(x)=Z(P)\cup xZ(P).$ Therefore $|C_P(x)|=2|Z(P)|$. Since $|P'|=2$
we have $|x^P|=2$ where $x\in P\setminus Z(P).$ Hence $|x^P|=2=|P|/|C_P(x)|=|P|/2|Z(P)|,$ where $x\in P\setminus Z(P).$
So $|P|=4|Z(P)|.$ Now $G=A\times P$ implies that $|G|=4|Z(G)|.$
\end{proof}

\begin{lemma}\label{L12}
Let $G$ be a non-abelian group and $Z(G)\cong \mathbb{Z}_t.$ Then the following are equivalent:
\begin{itemize}
\item[1)] $[x]_{\sim_1}=[x]_{\sim_2}$ for all $x\in G.$
\item[2)] $G/Z(G)$ is an elementary abelian 2-group.
\item[3)] $G=A\times P$ and $|P'|=2$, $P$ is a 2-group, $A\cong \mathbb{Z}_m, m$ is odd and $Z(P)$ is cyclic.
\end{itemize}
\end{lemma}

\begin{proof}
$1) \Rightarrow 2)$  It is an immediate consequence of Lemma \ref{L5}.

$2) \Rightarrow 3)$  Let $G/Z(G)$ be an elementary abelian 2-group. Then $G'$ is an
elementary abelian 2-group. Since $Z(G)$ has only one element of order 2
$G'=\{1,t\}$ and hence $|P'|=2.$ From $x^G\subseteq xG',$ we find that $|x^G|=2,$ for all $x\in G\setminus Z(G)$ and
$\Gamma(G)$ is regular. By Lemma \ref{L2} $G=A\times P,$ where $A$ is an abelian group and $P$ is a $p$-group.
Since $G/Z(G)$ is a 2-group, $P/Z(P)$ is also a 2-group.
Therefore $P$ is a 2-group. also $Z(G)=A\times Z(P)$ is cyclic which implies that $A\cong \mathbb{Z}_m,$ where $m$
is odd and $Z(P)$ is cyclic.

$3) \Rightarrow 1)$ Since $|G'|=2$, by Lemma \ref{L9}, $[x]_{\sim_1}=[x]_{\sim_2}$ for all $x\in G.$
\end{proof}

\begin{proposition}\label{P1}
Let $G$ be a $\CA$-group and $|Z(G)|=p$ where $p$ is a prime. Then $[x]_1=[x]_2$ for all $x\in G$ if and only if $G\cong D_8$ or $Q_8.$
\end{proposition}

\begin{proof}
Let $[x]_{\sim_1}=[x]_{\sim_2}.$ By Lemma \ref{L5} $G/Z(G)$ and $G'$ are both elementary abelian 2-group.
Therefore $|Z(G)|=|G'|=2$ and $G$ is an extra-special 2-group. Since $G$ is a $\CA$-group and $[x]_{\sim_1}=[x]_{\sim_2}$
we find by Lemma \ref{L11}, that $|G|=4|Z(G)|$. Hence $|G|=8.$ Since the extra-special
groups of order $8$ are $D_8$ and $Q_8$, we have $G\cong D_8$ or $Q_8$.

Conversely $D_8$ and $Q_8$ are $\CA$-groups and $[x]_{\sim_1}=[x]_{\sim_2}$ for all $x\in D_8, Q_8$.
\end{proof}

\begin{lemma}\label{L13}
Let $G_1$ and $G_2$ be two groups. Let $[g_1]_{\sim_1}=[g_1]_{\sim_2}$, for all $g_1\in G_1$
and $[g_2]_{\sim_1}=[g_2]_{\sim_2}$ for all $g_2\in G_2$. Then $[X]_{\sim_1}=[X]_{\sim_2}$, for all $X\in G_1\times G_2$.
\end{lemma}

\begin{proof}
Let $[g_1]_{\sim_1}=[g_1]_{\sim_2}$ for all $g_1\in G_1$ and $[g_2]_{\sim_1}=[g_2]_{\sim_2}$ for all
$g_2\in G_2$. Let $Y\in [X]_{\sim_1}$ where
$X, Y\in G_1\times G_2$. So there exist $a_1, b_1\in G_1$ and $a_2, b_2\in G_2$ such that $X=(a_1, a_2)$
and $Y=(b_1, b_2).$ Then $C_{G_1\times G_2}(a_1, a_2)=C_{G_1\times G_2}(b_1, b_2)$. Therefore
$C_{G_1}(a_1)\times C_{G_2}(a_2)=C_{G_1}(b_1)\times C_{G_2}(b_2).$ Moreover
$C_{G_1}(a_1)=C_{G_1}(b_1)$ and $C_{G_2}(a_2)=C_{G_2}(b_2).$
By assumption $a_1Z(G_1)=b_1Z(G_1)$ and $a_2Z(G_2)=b_2Z(G_2)$. So $a_1Z(G_1)\times a_2Z(G_2)=b_1Z(G_1)\times b_2Z(G_2)$.
Therefore $(a_1, a_2)(Z(G_1)\times Z(G_2))=(b_1, b_2)(Z(G_1)\times Z(G_2))$ and $Y\in [X]_{\sim_2}$.
Hence $[X]_{\sim_1}\subseteq [X]_{\sim_2}$. Now by Lemma \ref{L6}, $[X]_{\sim_1}=[X]_{\sim_2}$ for all $X\in G_1\times G_2.$
\end{proof}

\begin{proposition}\label{P2}
Let $G$ be a non-abelian group. Let $[x]_{\sim_1}=[x]_{\sim_2},$ for all $x\in G.$
Then $C_G(x)\unlhd G,$ for all $x\in G.$
\end{proposition}

\begin{proof}
Let $a\in C_G(x)$ and $g\in G.$ Since $[x]_{\sim_1}=[x]_{\sim_2},$ for all $x\in G$, by
Lemma \ref{L5}, $G/Z(G)$ is an elementary abelian 2-group. So $g^{-1}agZ(G)=aZ(G)$.
Therefore there exists $z\in Z(G)$ such that $g^{-1}ag=az$. Since $az\in C_G(x)$, we find that $g^{-1}ag\in C_G(x)$.
Hence $C_G(x)\unlhd G$, for all $x\in G$.
\end{proof}

\begin{proposition}\label{P3}
Let $G$ be a non-abelian group. Then $G$ is a $\CA$-group and $[x]_{\sim_1}=[x]_{\sim_2},$ for all $x\in G$ if and only if
$|G|=\frac{2|Z(G)|^2}{(3|Z(G)|-k)}$, where $k$ is the number of conjugacy classes of $G$. In particular
if $G$ is a $\CA$-group and $[x]_{\sim_1}=[x]_{\sim_2},$ for all $ x\in G$, then $|G|\leq 2|Z(G)|^2.$
\end{proposition}
\begin{proof}
Let $G$ be a $\CA$-group and $[x]_{\sim_1}=[x]_{\sim_2},$ for all $x\in G$. Since there are $|Z(G)|^2$ edges
between two classes $xZ(G)\neq yZ(G)$, we have $|E(\Gamma(G))|=(2^{\frac{|G|}{Z(G)}-1})|Z(G)|^2$. By \cite[Lemma 3.27]{AAM},
$|E(\Gamma(G))|=\frac{|G|^2-k|G|}{2}$. Therefore $|G|=\frac{2|Z(G)|^2}{3|Z(G)|-k}$.

Conversely let $|G|=\frac{2|Z(G)|^2}{(3|Z(G)|-k)}$. So $|G|=|Z(G)|+(k-|Z(G)|)\frac{|G|}{2|Z(G)|}$.
On the other hand for all $x\in G\setminus Z(G)$, $|x^G|\leq \frac{|G|}{2|Z(G)}$. Hence
$|x^G|=\frac{|G|}{2|Z(G)|}$, for all $x\in G\setminus Z(G)$. So $|C_G(x)|=2|Z(G)|$, for
all $x\in G\setminus Z(G)$. Now by Lemma \ref{L10}, $G$ is a $\CA$-group and $[x]_{\sim_1}=[x]_{\sim_2}.$
\end{proof}
\begin{proposition}\label{P4}
Let $G$ be a $\CA$-group and $[x]_{\sim_1}=[x]_{\sim_2}$, for all $x\in G$. Assume that
$[x,y]\neq 1$ for $x,y\in G\setminus Z(G)$, then there exists one and only one $[w]_{\sim_2}$ so that
$[w]_{\sim_2}\neq [y]_{\sim_2}$ and $x^y=x^w$. Moreover if $[w]_{\sim_2}\neq [y]_{\sim_2}$ for $w\in G$
and $x^y=x^w$, then $[w]_{\sim_2}=[xy]_{\sim_2}$.
\end{proposition}
\begin{proof}
Let $G$ be a $\CA$-group and $[x]_{\sim_1}=[x]_{\sim_2}$, for all $x\in G$. Clearly $x^{y}=x^{xy}$. Next we show
that $[w]_{\sim_2}\neq [y]_{\sim_2}$. Suppose by contrary that $[w]_{\sim_2}=[y]_{\sim_2}$. Then $yZ(G)=xyZ(G)$ and
$x\in Z(G)$, which is a contradiction.

Let $[w]_{\sim_2}\neq [y]_{\sim_2}$ for $w\in G$ and $x^y=x^w$. Since $x^y=x^w$, we have $wy^{-1}\in C_G(x)$ and
from $[w]_{\sim_2}\neq [y]_{\sim_2}$ we find that $wZ(G)\neq yZ(G)$, which means $wy^{-1}\notin Z(G)$.
By assumption $G$ is a $\CA$-group and $[x]_{\sim_1}=[x]_{\sim_2}$, for all $x\in G$. This implies by Lemma \ref{L10},
$C_G(x)=Z(G)\cup xZ(G)$ for $x\in G\setminus Z(G)$. Hence $wy^{-1}\in xZ(G)$, $wZ(G)=xyZ(G)$ and $[w]_{\sim_2}=[xy]_{\sim_2}$.
\end{proof}
In the following we show that if $G$ is a $\CA$-group and $|\cent(G)|=|G|/|Z(G)|$, then
there exists an integer $r>1$ such that $|\cent(G)|=2^r.$

\begin{proposition}\label{P5}
Let $G$ be a $\CA$-group and $[x]_{\sim_1}=[x]_{\sim_2},$ for all $x\in G$. Then there exist
subgroups $H_i\unlhd G, i=1,...,r$ such that $Z(G)\leq H_1\leq ...\leq H_r=G$, $|H_i|=2^i|Z(G)|, i=1,...,r$;
$Z(H_i)=Z(G), i\geq 2$ and $[h]_{\sim_1}=[h]_{\sim_2}$, for all $h\in H_i, i=2,...,r.$
\end{proposition}
\begin{proof}
Since $G$ is non-abelian, so there exists $x,y\in G$ such that $[x,y]\neq 1$. Let $H_1=Z(G)\langle x\rangle$.
Then $H_1\unlhd G,$ is abelian and $|H_1|=2|Z(G)|$. Let $H_2=H_1\langle y\rangle$. Then
$H_2\unlhd G, Z(H_2)=Z(G), |H_2|=2^2|Z(G)|$ and $[x]_{\sim_1}=[x]_{\sim_2},$ for all
$x\in H_2$. If $G\neq H_2$, then there exist $u\in G\setminus H_2$. Let $H_3=H_2\langle u\rangle$.
$H_3\unlhd G$, $Z(H_3)=Z(G)$, $|H_3|=2^3|Z(G)|$ and $[x]_{\sim_1}=[x]_{\sim_2}$,
for all $x\in H_3$. Since $G$ is finite, there exists a positive integer $r$ such that
$H_r=G$, $|G|=2^r|Z(G)|$. Hence $|\cent(G)|=|G|/|Z(G)|=2^r.$

\end{proof}

\subsection{Proof of Theorem \ref{T1}.}
\begin{proof}
$1) \Rightarrow 2)$  By Lemma \ref{L10}, $C_G(x)=Z(G)\cup xZ(G)$. Hence for
all $x\in G\setminus Z(G)$, $|C_G(x)|=2|Z(G)|$ and $\Gamma(G)$ is regular. So by Lemma \ref{L2}, $G=A\times P$,
where $A$ is an abelian group, $P$ is a $p$-group, $p$ a prime. By Lemma \ref{L5}, $G/Z(G)$
is an elementary abelian 2-group. Therefore $P$ is a 2-group and $P/Z(P)$ is an elementary abelian 2-group. Since
$C_G(x)=Z(G)\cup xZ(G)$, it follows for each $(a,w)\in A\times P$, we have
\begin{equation*}
\begin{split}
A\times C_P(w)&=C_{A\times P}(a,w)=Z(A\times P)\cup (a,w)Z(A\times P)\\
&=(A\times Z(P))\cup(A\times wZ(P))=A\times (Z(P)\cup wZ(P)).
\end{split}
\end{equation*}
Hence $C_P(w)=Z(P)\cup wZ(P)$.
By Lemma \ref{L10}, $P$ is a $\CA$-group and $[x]_{\sim_1}=[x]_{\sim_2}$ for all $x\in P.$

$2) \Rightarrow 3)$ Since $P$ is a $\CA$-group and $[x]_{\sim_1}=[x]_{\sim_2}$ for all $x\in P$
by Lemma \ref{L10}, $C_P(x)=Z(P)\cup xZ(P)$ for all $x\in P\setminus Z(G)$.

$3) \Rightarrow 1)$  By Lemma \ref{L10}, $P$ is a $\CA$-group and $[x]_{\sim_1}=[x]_{\sim_2}$
for all $x\in P$. Therefore by Lemma \ref{L13}, $G=A\times P$ is a $\CA$-group and $[x]_{\sim_1}=[x]_{\sim_2}$ for all $x\in G$.
\end{proof}

\subsection{Proof of Theorem \ref{T2}.}
\begin{proof}
1) By Lemma \ref{L12}, $G=A\times P$ where $A\cong\mathbb{Z}_m$, $m$ is odd and $P$ is a 2-group,
$Z(P)$ is cyclic and $|P'|=2.$ Therefore $P$ is of nilpotency class two, with cyclic center.
By Lemma \ref{L3}, $P$ has either the central decomposition:
\begin{align*} Q(n_1,r_1)...Q(n_{\alpha}, r_{\alpha}) Q(\ell,\ell)^{\epsilon_{\ell}}...Q(1,1)^{\epsilon_1}, \text{ where } &\alpha \geq 0, \epsilon_i \geq 0, i=1,...,\ell, \\n_1>...> n_{\alpha} > \ell\geq 1, n_{\alpha} > r_1 >...> r_{\alpha} \geq 0, &1 < n_1 - r_1< ...< n_{\alpha}-r_{\alpha}\end{align*}
or
\begin{align*} R(n) Q(\ell,\ell)^{\epsilon_{\ell}} ... Q(1,1)^{\epsilon_1} \text{ where } n\geq \ell\geq 1, \epsilon_{i} \geq 0, i=1,...,\ell.\end{align*}
In the first case either $n\geq 2r$ or $r\leq n <2r$. Since $|P'|=2$ so in $Q(n,r)$, $[a,b]^{2}=1$. Therefore
for $2r\leq n, (a^{2^{n-r}})^2=1$. Hence $a^{2^{n-r+1}}=1$. Since $a$ has order $2^n$, so $n\leq n-r+1$.
Therefore $r\leq 1$. Hence $Q(n,r)\cong Q(n,0)$, where $n>0$ or $Q(n,r)\cong Q(n,1)$, where $n>1$.
Also for $r\leq n <2r,$ $a^{2^r}=([a,b]^{2^{2r-n}}).$ Since $0< 2r-n$ and $[a,b]^2=1$ so $a^{2^r}=1.$
Since $o(a)=2^n$ therefore $n\leq r$. Hence $n=r.$ On the other hand $[a,b]$ has order $2^r$, and therefore $r=0$
or $r=1$. Since $n=r$ and $n>0$ so $r=1$. Hence $Q(n,r)=Q(n,n)=Q(1,1)$.
Hence central decomposition in this case is as $Q(n_1,1)Q(n_2,0)Q(1,1)^{\epsilon}$ or $Q(n_1,1)Q(1,1)^{\epsilon}$
or $Q(n_2,0)Q(1,1)^{\epsilon}$ or $Q(1,1)^{\epsilon}$. In case we have the central decomposition
$Q(n_1,1)Q(n_2,0)Q(1,1)^{\epsilon}$ then $0<n_1-1<n_2-0$ and $n_2<n_1$ which is a contradiction.
Hence $P$ has the following central decomposition:
$Q(n,1)Q(1,1)^{\epsilon}$ or $Q(n,0)Q(1,1)^{\epsilon}$ or $Q(1,1)^{\epsilon}$ where $\epsilon \geq 0 , n\geq 1$.
In $Q(n,1)Q(1,1)^{\epsilon}$ since $1<n-1$ it follows $2<n$. In $Q(n,0)Q(1,1)^{\epsilon}$ if $\epsilon=0$ then
$Q(n,0)Q(1,1)^{\epsilon} \cong  Q(n,0)$ which is abelian group and is a contradiction so $\epsilon \geq 1$. it is clear that
in this case $n>1$.
In Case (2) for $R(n)$, $a^{2^{n}}\neq 1.$ So $[a,b]^{2^{n-1}}\neq 1.$ Therefore $n-1=0$ and $n=1$.
Since $n\geq \ell\geq 1$ so $\ell=1$, $P$ has the central decomposition $R(1)Q(1,1)^{\epsilon}, \epsilon \geq 0$.

2) By Lemma \ref{L12}, $G=A\times P$ where $P$ is a 2-group, $A\cong \mathbb{Z}_{m}$, $m$ is odd,
$Z(P)$ is cyclic and $|P'|=2.$ By Lemma \ref{L11}, $|P|=4|Z(P)|$, therefore $P$ has
at most 3 generators. Hence $P$ has one of the following central decomposition
$Q(n,0)Q(1,1), n>1$ or $Q(n,1), n>2$ or $Q(1,1)$ or $R(1)$.
\end{proof}
\subsection{Proof of Theorem \ref{T3}.}
\begin{proof}
By Proposition \ref{P5} there exist an integer $r>1$ such that $m=|G|/|Z(G)|=2^{r}$.

Conversely let $r>1$ be an integer. By Lemma \ref{L4}, there exist a group $G$ such that $|G|=2^{r}$, $G/Z(G)$
is an elementary abelian $2$-group and $C_G(x)=Z(G)\langle x\rangle$. By \ref{L10}, $G$ is a $\CA$-group and
$[x]_{\sim_1}=[x]_{\sim_2}$ for all $x \in G$. Hence $G$ is an $m$-centralizer group where $m=|G|/|Z(G)|$.
\end{proof}
\subsection{Proof of Theorem \ref{T4}.}
\begin{proof}
1) Suppose that $G$ is not a $\CA$-group and $[x]_{\sim_1}=[x]_{\sim_2},$ for all $x\in G.$ Then there exist $x,y\in G\setminus Z(G)$
such that $[x,y]=1$ and $C_G(x)\neq C_G(y)$. Therefore there exists $w\in G\setminus Z(G)$ so that $[w, x]=1$ and $[w, y]\neq 1$.
Since $\diam(\Gamma(G))=2$, there exists $a\in G\setminus Z(G)$ such that $[a, x]\neq 1$ and $[a, y]\neq 1$.
By Lemma ~\ref{L5}, $G/Z(G)$ is an elementary abelian 2-group. Therefore all classes $[x]_{\sim_2},[y]_{\sim_2},[w]_{\sim_2},[a]_{\sim_2},[xy]_{\sim_2},[xw]_{\sim_2},[xa]_{\sim_2},[yw]_{\sim_2},[ya]_{\sim_2},$\\
$[wa]_{\sim_2}, [xyw]_{\sim_2},[xya]_{\sim_2},[xwa]_{\sim_2},[ywa]_{\sim_2},[axyw]_{\sim_2}$ are distinct.

Let $H=Z(G)\cup[x]_{\sim_2}\cup[y]_{\sim_2}\cup[w]_{\sim_2}\cup[a]_{\sim_2}\cup[xy]_{\sim_2}
\cup[xw]_{\sim_2}\cup[xa]_{\sim_2}\cup[yw]_{\sim_2}\cup[ya]_{\sim_2}\cup[wa]_{\sim_2}\cup[xyw]_{\sim_2}
\cup[xya]_{\sim_2}\cup[xwa]_{\sim_2}\cup[ywa]_{\sim_2}\cup[axyw]_{\sim_2}.$

Therefore $H\unlhd G$ and $|H|=16|Z(G)|$. All above classes dose not commute with at least one classes. So $Z(H)=Z(G)$.

2) Let $|G'|=2$. By Lemma \ref{L9} there exist $x,y\in G$ such that $[x,y]\neq 1$ and
$G=C_G(x)\cup C_G(y)\cup C_G(xy)$. By \cite[Lemma 4]{BBM}, $C_G(x)C_G(y)=G$. On the other
hand $|x^{G}|=|G|/|C_G(x)|=2$. Therefore $|C_G(x)|=\frac{|G|}{2}$ and
$|G|=|C_G(x)C_G(y)|=|C_G(x)||C_G(y)|/|C_G(x)\cap C_G(y)|=(\frac{|G|}{2})(\frac{|G|}{2})/|C_G(x)\cap C_G(y)|$.
Therefore $|C_G(x)\cap C_G(y)|=\frac{|G|}{4}$. By \cite[Lemma 2]{BBM}, $C_G(x)\cap C_G(y)=C_G(x)\cap C_G(y)\cap C_G(xy)$.
Let $K=C_G(x)\cap C_G(y)\cap C_G(xy)$. Therefore $|G|=4|K|$. We claim that $Z(K)=Z(G)$. Suppose by contrary that $Z(K)\neq Z(G)$.
Let $z_k \in Z(K)\setminus Z(G)$. There exist $g\in G\setminus K$ so that $gz_k\neq z_kg$. There are three following cases for $g$.

Case i) $g\in C_G(x)\setminus K.$

Case ii) $g\in C_G(y)\setminus K$ and

Case iii) $g\in C_G(xy)\setminus K$.

First consider Case i).
Let $g\in C_G(x)\setminus K$. By \cite[Lemma 7]{BBM}, $gK=C_G(x)\setminus K$. Let $t\in C_G(x)\setminus K$.
Then there exits $k\in K$ such that $t=gk$. Hence $tz_k=(gk)z_k=g(z_kk)=(gz_k)k\neq (z_kg)k=z_k(gk)=z_kt$.
This implies that $C_G(z_k)\cap [C_G(x)\setminus K]=\phi$. Now there are three subcases as follows:

Subcase a) $z_k$ commutes with all elements of
$(C_G(y)\cup C_G(xy))\setminus K$. We have $C_G(z_k)=[C_G(y)\cup C_G(x)]\setminus K\cup K=C_G(y)\cup C_G(xy)$.
A simple calculation shows that $|C_G(z_k)|=|C_G(y)|+|C_G(xy)|-|C_G(y)\cap C_G(xy)|=\frac{3|G|}{4}$, which is
a contradiction.

Subcase b) There exists an element $h\in C_G(y)\setminus K$ such that $z_kh\neq hz_k$. Hence either $z_k$
commutes with all elements of $C_G(xy)\setminus K$ or does not commute with at least one element of $C_G(xy)\setminus K$.
If $z_kg=gz_k$ for all $g\in C_G(xy)\setminus K$, then $C_G(z_k)=K\cup [C_G(xy)\setminus K]=C_G(xy)$.
Therefore $[z_k]_{\sim_1}=[xy]_{\sim_1}$. By Lemma \ref{L9}, $[z_k]_{\sim_2}=[xy]_{\sim_2}$.
Thus $z_kZ(G)=xyZ(G)$. This implies that there exists $z\in Z(G)$ such that $xy=z_kz$. Hence $(xy)x=(z_kz)x=x(z_kz)=x(xy)$
and therefore $xy=yx$, which is a contradiction.

Subcase c) There exists $g\in C_G(xy)\setminus K$ such that $gz_k\neq z_kg$.
It is easy to see that $z_kg\neq gz_k$ for all $g\in C_G(xy)\setminus K$. Therefore $C_G(z_k)=K$ and
$\frac{|G|}{2}=|C_G(z_k)|=|K|=\frac{|G|}{4}$ which is a contradiction. Finally a similar argument shows that the assumption
that $z_k$ dose not commute with at least one element of $C_G(xy)\setminus K$ gives another contradiction.

This complete the proof in Case i) and similar arguments can be used for the other two cases. Hence $Z(K)=Z(G)$.

For completing the proof it is enough to show that $s>1$. If $K$ is abelian, then
$K=Z(G)$ and we have $|G|=4|Z(G)|$. This implies that $G$ is a $\CA$ group which is not the case.
So $K$ is not abelian. Since $|K'|\leq |G'|=2$ so $|K'|=2$. Substituting $G$ by $K$ in this case, we find a
subgroup $K_2\subseteq K$ such that $|K|=4|K_2|$. Set $K_1=K$. By continuing we obtain a series of subgroups
$Z(G) \subseteq ...\subseteq K_2 \subseteq K \subseteq G$ in such a way that $|K_i|=4|K_{i+1}|$ and $K_1=K$.
Since $G$ is finite there exists a positive integer $s$ such that $K_s=Z(G)$ and $|G|=4^s|Z(G)|=2^{2s}|Z(G)|$.
If $s=1$ then $|G|=4|Z(G)|$ and by Lemma \ref{L11}, $G$ is a $\CA$-group, which is not the case. Hence $s>1$.
\end{proof}

\section*{Acknowledgement}
The authors thanks Research Deputy of Yazd University for some financial support.

\bibliographystyle{abbrvnat}
\def\cprime{$'$}

\bigskip
\bigskip
{\footnotesize \pn{\bf Mohammad~A.~Iranmanesh}\; \\ {Department of
Mathematics},\\ {Yazd University, 89195-741,} {Yazd, Iran}\\
{\tt Email: iranmanesh@yazd.ac.ir}\\

{\footnotesize \pn{\bf Mohammad Hossein Zareian}\; \\ {Department of
Mathematics},\\ {Yazd University, 89195-741,} {Yazd, Iran}\\
{\tt Email: math01396@gmail.com}\\

\end{document}